\documentclass{amsart}
\usepackage{amsfonts}
\usepackage{amsthm}
\usepackage{amsmath}
\usepackage{amssymb}
\usepackage{amscd}
\usepackage{cite}
\usepackage[latin2]{inputenc}
\usepackage{t1enc}
\usepackage{mathrsfs}
\usepackage{indentfirst}
\usepackage{graphicx}
\usepackage{graphics}
\usepackage{hyperref}
\usepackage{pict2e}
\usepackage{epic}
\usepackage{tikz-cd}
\numberwithin{equation}{section}
\usepackage{epstopdf}

\newenvironment{spmatrix}
  {\left(\begin{smallmatrix}}
  {\end{smallmatrix}\right)}
\newtheorem{thm}{Theorem}[section]
\newtheorem{cor}[thm]{Corollary}
\newtheorem{lem}[thm]{Lemma}
\newtheorem{prop}[thm]{Proposition}

\theoremstyle{definition}

\newtheorem{defn}[thm]{Definition}
\theoremstyle{remark}
\newtheorem{rem}[thm]{Remark}

\newtheorem{note}[thm]{Note}

\def\A{\mathbb{A}}
\def\C{\mathbb{C}}

\def\Q{\mathbb{Q}}
\def\R{\mathbb{R}} 
\def\Z{\mathbb{Z}}

\def\begcd{\begin{tikzcd}}
\def\endcd{\end{tikzcd}}
\def\begenum{\begin{enumerate}}
\def\endenum{\end{enumerate}}
\def\begpmat{\begin{pmatrix}}
\def\endpmat{\end{pmatrix}}
\def\a{\alpha}

\def\e{\varepsilon}
\def\p{\prime}
\def\g{\gamma}
\def\s{\sigma}

\def\k{\kappa}

\def\O{\mathcal{O}}
\def\fa{\mathfrak{a}}
\def\fb{\mathfrak{b}}
\def\fc{\mathfrak{c}}
\def\fd{\mathfrak{d}}

\def\fp{\mathfrak{p}}

\def\fN{\mathfrak{N}}

\newcommand{\Hom}{{\rm{Hom}}}
\newcommand{\ovl}{\overline}
\newcommand{\Tr}{\operatorname{Tr}}
\newcommand{\id}{\operatorname{id}}
\newcommand{\co}{\operatorname{c}}

\newcommand{\GL}{\operatorname{GL}}

\begin{document}

\title[Special $L$-values of cusp forms on GL(2)]{Non-vanishing of special $L$-values of cusp forms on GL(2) with totally split prime power twists}
\date{\today}
\author[Jaesung Kwon]{JAESUNG KWON}\author[Hae-Sang Sun]{HAE-SANG SUN}
\address{Ulsan National Institute of Science and Technology \\ Ulsan \\ Korea}\email{jaesung.math@gmail.com}\address{Ulsan National Institute of Science and Technology \\ Ulsan \\ Korea}\email{haesang@unist.ac.kr}


\begin{abstract} 
We prove the non-vanishing of special $L$-values of cuspidal automorphic forms on GL(2) twisted by Hecke characters of prime power orders and totally split prime power conductors. Main ingredients of the proof are estimating the Galois averages of fast convergent series expressions of special $L$-values and considering Shintani cone decomposition.
\end{abstract}

\maketitle
\setcounter{tocdepth}{1} 
\tableofcontents

\section{Introduction}

Studying the non-vanishing of special $L$-values is interesting as fascinating mathematical techniques and ingredients are involved. There are plenty of non-vanishing results from which important arithmetic consequences arise.

Rohrlich \cite{rohrlich1984onl} shows the non-vanishing of the special values of cyclotomic modular $L$-values by estimating the Galois averages of fast convergent series expressions of the special $L$-values. As a consequence, Rohrlich \cite{rohrlich1984onl} deduces that the Mordell-Weil groups of CM elliptic curves over the cyclotomic $\Z_p$-extension of $\Q$ are finitely generated. 

One can consider more general settings. 
Rohrlich \cite{rohrlich1989nonvanishing} proves that the $L$-values of automorphic forms on $\GL(2)$ are non-vanishing for infinitely many GL(1) twists by modifying the idea of Shintani cone decomposition \cite{shintani1979aremark} to count the number of ideals with bounded norm in an arithmetic progression. 
Van Order \cite{van2019rankin} gives a proof of the non-vanishing of cyclotomic $L$-values of non-dihedral automorphic forms on GL(2) over totally real fields.

The main result of the present paper is to show the non-vanishing of the special $L$-values of automorphic forms on $\GL(2)$ over general number fields twisted by Hecke characters of $p$-power order and totally split prime power conductors, which is a generalization of Rohrlich \cite[Theorem]{rohrlich1984onl} or Luo and Ramakrishnan \cite[Proposition 2.2]{luo1997determination}.

\subsection{Main Theorems}\label{Main:Theorems} Let us give some notations and settings. Let $F$ be a number field, $\fN$ an integral ideal of $F$, $S_{(\mathbf{k},\mathbf{m}),J}(\fN,\chi)$ the space of cuspidal automorphic forms over $F$ of a cohomological weight $(\mathbf{k},\mathbf{m})$, type $J$ and level $\fN$ with a central character $\chi$. Let $\frac{k}{2}\in\frac{1}{2}\mathbb{Z}_{>0}$ be the central critical point for $f$ in the arithmetic normalization. Let $\fp$ be a prime ideal of $F$ lying above an odd prime number $p$.
Let $f\in S_{(\mathbf{k},\mathbf{m}),J}(\fN,\chi)$ be a newform, $\psi$ be a Hecke character of $p$-power order with $\fp$-power conductors, $L(s,f\otimes\psi)$ the $L$-function attached to $f$ and $\psi$, $K_f$ the Hecke field of $f$ over $\Q$, which is the field adjoining all the Fourier-Whittaker coefficients of $f$ to $\Q$, and $n_0:=\textrm{max}\{m\in\Z|\ \mu_{p^m}\subset K_f\}$.

To obtain fast convergent series expressions of twisted modular $L$-values, namely the approximate functional equation of the $L$-function, we have to find the Fourier-Whittaker expansions of automorphic forms and its Mellin transforms. Also we need Atkin-Lehner theory for automorphic forms of GL(2) over $F$. These ingredients will be developed in Section \ref{cuspform} and \ref{sp:Lvalue} together with a brief introduction to automorphic forms of GL(2) over $F$.

In Section \ref{esti:arith:progress}, we estimate the number of ideals of bounded norm and the lower bound of absolute norms of elements in an arithmetic progression, which make it possible to calculate the limit of the Galois averages for our case. These are achieved by using the coherent cone decomposition (see Rohrlich \cite{rohrlich1989nonvanishing}). 

Define the Galois averages of special $L$-values by
$$
L_{\textrm{av}}(f\otimes\psi):=\frac{1}{[K_f(\psi):K_f]}\sum_{\s\in\text{Gal}(K_f(\psi)/K_f)} L\Big(\frac{k}{2},f\otimes\psi^\s\Big) .
$$
In Section \ref{hecke:char} and Section \ref{estimation}, we will discuss the Galois averages of twisted $L$-values and obtain their estimations, which play the key role in the proofs of main theorems. We would like to remark that for a prime $\fp$ of $F$ which does not split totally, the corresponding estimations of the Galois averages over the characters are worse than ones in the present paper, due to the existence of the corank one group of $p$-adic units.

We set $\theta\in[0,\frac{1}{2}]$ as a bound of an exponent of the eigenvalues of a Hecke eigenform.
Let us assume that $\fp$ is a totally split prime and coprime to $h_F\fd_F\fN$. Then we obtain an estimation of the Galois averages of special $L$-values which depends on $\theta$:
\begin{thm}\label{main:thm:0:intro} Suppose some parity condition (see (\ref{parity:cond}) and Remark \ref{parity:cond:rem}) on $F$, $(\mathbf{k},\mathbf{m})$ and $J$. If we have $\theta$ less than $\frac{1}{4}$ ($\frac{1}{6}$ for the case of $p=3$, $\frac{1}{4}$ otherwise), an estimation of $L_{\operatorname{av}}(f\otimes\psi_n)$ is given by
\begin{align}\begin{split}
L_{\operatorname{av}}(f\otimes\psi_n)=1+o(1)
\end{split}\end{align}
as $n$ tends to the infinity.
\end{thm}

Note that the more optimal $\theta$ we have, the better estimation on the error of the Galois averages we have. 

Let $(\psi_n)_n$ be a primitive element of $\Hom_{\text{cont}}(\text{Cl}(F,\fp^\infty),\mu_{\infty})$.
By applying Theorem \ref{main:thm:0:intro} together with the algebraicity result of Hida \cite{hida1994critical} and the bound $\theta=7/64$, which is obtained by Blomer-Brumley \cite{blomer2011ramanujan} and Nakasuji \cite{nakasuji2012generalized}, we obtain the following non-vanishing result:

\begin{thm}\label{main:thm:1:intro} Suppose the parity condition on $F$, $(\mathbf{k},\mathbf{m})$ and $J$. Then we have
$$
L\Big(\frac{k}{2},f\otimes\psi\Big)\neq 0
$$
for almost all Hecke characters $\psi$ over $F$ of $p$-power order and $\fp$-power conductor.
\end{thm}

Kim-Sun \cite{kim2017modular} obtain partial results toward the mod $p$ non-vanishing of cyclotomic modular $L$-values by studying the homological nature of modular symbols. Sun \cite{sun2018generation} proves that the Hecke field of a newform can be generated by a single special value of the modular $L$-function, by studying an additive variant of the Galois average, which supersedes the non-vanishing results. The first named author plans to generalize the results in Kim-Sun \cite{kim2017modular} and Sun \cite{sun2018generation} to the current setting.

\subsection{Notations}

Let us provide some notations which will be used globally in this paper. Let $F$ be a number field, $\A_F$ the adele ring of $F$, $\A_F^{(\infty)}$ the finite adele of $F$, $F_\infty:=F\otimes_\Q \R$ the infinite adele of $F$, $\O_F$ the integer ring of $F$, $\fd_F$ the different ideal of $F/\Q$, $d_F$ a finite idele of $\fd_F$, $D_F$ the discriminant of $F$, $h_F$ the class number of $F$, and $|\cdot|_{\A_F}=|\cdot|_{F_\infty}|\cdot|_{\A_F^{(\infty)}}$ the idelic norm. For any place $v$ of $F$, denote by $\O_{F,v}$ the completion of $\O_F$ with respect to $v$. Let $\widehat{A}:=A\otimes_\Z\widehat{\Z}$ for a $\Z$-algebra $A$. Let $p$ be an odd prime number coprime to $D_F$, $\zeta_m$ a primitive $p^m$-th root of unity and $\mu_m$ the set of $p^m$-th roots of unity. 
Let us set by $N$ the norm map of $F/\Q$ (or $\prod_{\fp|p}N_{F_\fp/\Q_p}$) and
by $\Tr$ the trace map of $F/\Q$ (or $\sum_{\fp|p}\Tr_{F_\fp/\Q_p}$). For $z\in\C$, denote $\textbf{e}(z):=\textrm{exp}(2\pi i z)$.

\section{Cusp forms}\label{cuspform}
\subsection{Cusp forms on GL(2)}
In this section, we will briefly give the definition of cuspidal automorphic forms on $\GL_2(\A_F)$ of cohomological weight. From now on, these are called cusp forms for simplicity.
All the settings in this section come from Hida \cite{hida1994critical}.

Let us give some notations. Let $I_F:=\text{Gal}(F/\Q)$ and $\Z[I_F]$ the free $\Z$-module generated by $I_F$. Let $\id$ and $\co\in I_F$ be the identity map and complex conjugation on $F$ (or on $\C$), respectively. Let $\Sigma(\R)$ be the set of real places of $F$, $\Sigma(\C)$ the set of complex places of $F$, and $J$ a subset of $\Sigma(\R)$. Note that $\Sigma(\R)$ and $\Sigma(\C)$ can be considered as subsets in $I_F$. Let $\textbf{k}=\sum_{\s\in I_F}k_\s \s$ and $\textbf{m}=\sum_{\s\in I_F}m_\s \s$ be elements in $\Z[I_F]$ satisfying following conditions:
\begin{enumerate}
\item $k_\s\geq 2$,
\item $k_\s+2m_\s=k_\tau+2m_\tau \text{ for any }\s,\tau\in I_F$,
\item $k_\s=k_{\s\co} \text{ for any }\s\in\Sigma(\C).$
\end{enumerate}
Let $\textbf{t}=\sum_{\s\in I_F}\s$, $\textbf{n}=\textbf{k}-2\textbf{t}$, and $\textbf{n}^*=\sum_{\s\in I_F}n_\s^*\s\in\Z[I_F]$, where $n_\s^*=n_\s+n_{\s\co}+2$ for $\s\in\Sigma(\C)$ and $n_\s^*=0$ for $\s\in\Sigma(\R)\cup\Sigma(\C)\co$.
Let $\fN$ be a non-zero integral ideal of $F$ and define subgroups of $\GL_2(\widehat{\O}_F)$ by
\begin{align*}
U_0(\fN)&:=\bigg\{\begin{pmatrix} a & b \\ c & d \end{pmatrix}\in \GL_2(\widehat{\O}_F) : c\in \widehat{\fN} \bigg\} , \\
U_1(\fN)&:=\bigg\{\begin{pmatrix} a & b \\ c & d \end{pmatrix}\in U_0(\fN) : d\in 1+\widehat{\fN} \bigg\} .
\end{align*}
Let us denote $\mathbf{x}_\s=\begin{spmatrix} X_\s \\ Y_\s \end{spmatrix}$, $\mathbf{x}=\otimes_{\s\in\Sigma(\C)}\mathbf{x}_\s$ where $X_\s$, $Y_\s$'s are indeterminates. For a commutative ring $R$ with unity and $\mathbf{d}=\sum_{I_F}d_\s \s\in\Z[I_F]$, let $L(\mathbf{d},R):=\otimes_{\s\in I_F}L(d_{\s},R)$ where $L(d_{\s},R)$ is the space of homogeneous polynomials of variable $\mathbf{x}_\s$ of degree $d_\s$ with the coefficients in $R$. Note that there is a usual action of GL(2) on $L(\mathbf{d},R)$. 
Let $C_{F,\infty}^+$ be the maximal compact subgroup of SL$_2(F_\infty)$, which is given by
$$
C_{F,\infty}^+=\prod_{\s\in\Sigma(\R)}\operatorname{SO}_2(\R)\times\prod_{\s\in\Sigma(\C)}\operatorname{SU}_2(\C) .
$$

\begin{defn}[{Hida \cite[Section 2]{hida1994critical}}]\label{adel:cuspform:defn} Let $\chi:F^\times\backslash\A_F^\times\rightarrow\C^\times$ be a Hecke character of modulus $\fN$ satisfying $\chi_\infty(z_\infty)=z_\infty^{-(\textbf{n}+\textbf{2m})}$ for $z_\infty\in F_\infty^\times$. A cusp form on $\GL_2(\A_F)$ of weight $(\textbf{k},\textbf{m})$, a type $J$, a level $U_0(\fN)$, and a central character $\chi$, is a $C^\infty$-function $f:\GL_2(\A_F)\rightarrow L(\mathbf{n}^*,\C)$ such that
\begenum
\item[(1)] $D_\s f=\big(\frac{n_\s^2}{2}+n_\s\big)f$ for $\s\in I_F$, where $D_\s$ is the Casimir operator for $\GL_2$ corresponding to $\s$.
\item[(2)] $f(\g z g u)=\chi(z)\chi_\fN(u)f(g)$ for $\g\in \GL_2(F)$, $z\in \A_F^\times$, $g\in\GL_2(\A_F)$ and $u\in U_0(\fN)$ where $\chi_\fN\big(\begin{spmatrix} a & b \\ c & d \end{spmatrix}\big):=\prod_{v|\fN}\chi(d_v)$.
\item[(3)] $f(g u_{\infty})(\mathbf{x})=\textbf{e}\big(\sum_{\s\in J}k_\s\theta_\s-\sum_{\s\in\Sigma(\R)\backslash J}k_\s\theta_\s\big)f(g)\big(\otimes_{\s\in\Sigma(\C)}u_\s\mathbf{x}_\s\big)$
for $g\in \GL_2(\A_F)$ and
$$
u_\infty=\bigg(\bigg(\begin{pmatrix} cos(2\pi\theta_\s) & sin(2\pi\theta_\s) \\ -sin(2\pi\theta_\s) & cos(2\pi\theta_\s) \end{pmatrix}\bigg)_{\s\in\Sigma(\R)},(u_\s)_{\s\in\Sigma(\C)}\bigg)\in C_{F,\infty}^+ .
$$
\item[(4)] $\int_{F\backslash \A_F}f(\begin{spmatrix} u & 0 \\ 0 & 1 \end{spmatrix}g)(\textbf{s})du=0$ for $g\in \GL_2(\A_F)$ where $du$ is a Haar measure on $F\backslash\A_F$.
\endenum
From now on, we denote by $S_{(\textbf{k},\textbf{m}),J}(\fN,\chi)$ the space of the aforementioned cusp forms on $\GL_2(\A_F)$.
\end{defn}

Let us denote by $K_j$ the $j$-th modified Bessel function of the second kind, and write $y_\infty=(y_\s)_\s\in F_\infty^\times$. Let $W_{\textbf{k},\textbf{m}}:F_{\infty}^\times\rightarrow L(\mathbf{n}^*,\C)$ be the Whittaker function defined by 
$$
W_{\textbf{k},\textbf{m}}(y_\infty)(\mathbf{x}):=\prod_{\s\in\Sigma(\R)}W_{\textbf{k},\textbf{m},\s}(y_\s)\cdot\bigotimes_{\s\in\Sigma(\C)}W_{\textbf{k},\textbf{m},\s}(y_\s)(\mathbf{x}_\s)
$$ 
where $W_{\textbf{k},\textbf{m},\s}(y_\s):=|y_\s|^{-m_\s}\textbf{e}(i|y_\s|)$ for $\s\in\Sigma(\R)$ and
\begin{align*}
W_{\textbf{k},\textbf{m},\s}(y_\s)(\textbf{x}_\s):=&\sum_{j_\s=0}^{n_{\s}^*}\begin{spmatrix} n_{\s}^* \\ j_\s \end{spmatrix} y_\s^{-m_\s}\ovl{y}_\s^{-m_{\s\co}}\bigg(\frac{y_\s}{i|y_\s|}\bigg)^{n_{\s\co}+1-j_\s} \\
&\times K_{j_\s-1-n_{\s\co}}(4\pi|y_\s|)X_\s^{n_{\s}^*-j_\s}Y_\s^{j_\s}
\end{align*}
for $\s\in\Sigma(\C)$. Then we have the Fourier-Whittaker expansion of a cusp form on $\GL_2(\A_F)$:

\begin{prop}[Hida {\cite[Theorem 6.1]{hida1994critical}}]\label{four:whit:exp:hida}
Let $\mathscr{F}$ be the group of fractional ideals of $F$. For $f\in S_{(\mathbf{k},\mathbf{m}),J}(\fN,\chi)$, there exists a function $a_f:\mathscr{F}\rightarrow\C$ satisfying following properties:
\begenum
\item[(1)] $a_f(\fa)=0$ if $\fa\in\mathscr{F}$ is not integral. 
\item[(2)] We have the Fourier-Whittaker expansion of $f$ by
$$
f\bigg(\begpmat y & x \\ 0 & 1 \endpmat\bigg)(\mathbf{x})=|y|_{\A_F}\sum_{\xi\in F^\times,[\xi]=J}a_f(\xi y\fd_F)W_{\mathbf{k},\mathbf{m}}(\xi y_\infty)(\mathbf{x})\mathbf{e}_F(\xi x)
$$
for $x,y\in\A_F$, where $[\xi]:=\{\s\in\Sigma(\R):\xi^\s>0\}$, $y_\infty\in F_\infty$ is the infinite part of $y$, and $\mathbf{e}_F:\A_F/F\rightarrow\C^\times$ is the additive character such that $\mathbf{e}_F(z_\infty):=\prod_{\s\in\Sigma(\R)\cup\Sigma(\C)}\mathbf{e}(\operatorname{Tr}_{F_\s/\R}(z_\s))$ for $z_\infty=(z_\s)_\s\in F_\infty$.


\endenum

\end{prop}

\section{Approximate functional equations}\label{sp:Lvalue}

Let $\fN$ be a integral ideal of $F$, $\chi$ a Hecke character of modulus $\fN$ satisfying $\chi_\infty(z_\infty)=z_\infty^{-(\textbf{n}+\textbf{2m})}$ for $z_\infty\in F_\infty^\times$, and $f\in S_{(\textbf{k},\textbf{m}),J}(\fN,\chi)$. In this section, we obtain an approximate functional equation of the $L$-function $L(s,f)$ of a cusp form $f$ in a spirit of Luo-Ramakrishnan \cite{luo1997determination}.

Let $U_F:=F_\infty^\times\times\widehat{\O}_F^\times$ be the maximal compact subgroup of $\A_F^\times$. Let 
$F_{\infty,+}^\times:=\prod_{\s\in\Sigma(\R)}\R^\times_{>0} \times \prod_{\s\in\Sigma(\C)}\C^\times$ be the identity connected component of 
$F_\infty^\times$
.
We can fix a representative $\{a_i\}_{i=1}^{h_F}\subset\A_F^{(\infty),\times}$ of the class group $\text{Cl}(F)\cong F^\times\backslash\A_F^\times/U_F$ of $F$ such that the corresponding fractional ideals of $F$ are $\{\fa_i\}_{i=1}^{h_F}$. 
Let us set $t_i:=\begin{spmatrix} a_i & 0 \\ 0 & 1 \end{spmatrix}$, then by the strong approximation theorem, we have 
\begin{align}\label{strong:approx}
F^\times\backslash\A_F^\times\cong \coprod_{i=1}^{h_F} a_i\cdot\big(\O_F^\times\backslash U_F\big) 
\end{align}


Define a number $[\textbf{d}]\in\Z$ by $|a|^{2[\textbf{d}]}:=a^{\textbf{d}}a^{\textbf{d}\co}$ for $\textbf{d}\in\Z[I_F]$. In our case, $[\textbf{n}+2\textbf{m}]=n_\s+2m_\s$, which is independent on $\s\in I_F$ due to our assumption on the weight $(\mathbf{k},\mathbf{m})$. So from now on, we will write $k:=[\textbf{n}+2\textbf{m}]+2$.

Recall that the $L$-function $L(s,f)$ of $f$ is given by analytic continuation of the following Dirichlet series
\begin{align*}
\sum_{0\neq\fa<\O_F}\frac{a_f(\fa)}{N(\fa)^s} \text{ for }\mathfrak{R}(s)>\frac{k+2}{2}
\end{align*}
where $\fa$ runs over the set of nonzero integral ideals of $F$.


\subsection{Integral representation of special $L$-values} To obtain an approximate functional equation of the special $L$-value $L\big(\frac{k}{2},f\big)$, we need to compute the Mellin transform of $f$. 

For $\mathbf{j}=\sum_{\s\in\Sigma(\C)}j_\s \s\in\Z[I_F]$, we define the $\mathbf{j}$-th component $f_\mathbf{j}$ of a $L(\textbf{n}^*,\C)$-valued function $f$ by $f_\mathbf{j}(g):=\prod_{\s\in \Sigma(\R)}f_\s(g)\prod_{\s\in \Sigma(\C)}f_{\s,j_\s}(g)$ where $g\in\GL_2(\A_F)$ and
$$
f(g)(\mathbf{x})=\prod_{\s\in\Sigma(\R)}f_\s(g)\cdot\bigotimes_{\s\in\Sigma(\C)}\sum_{j_\s=0}^{n_\s^*} f_{\s,j_\s}(g) X_\s^{n_\s^*-j_\s}Y_\s^{j_\s} .
$$ 
By Proposition \ref{four:whit:exp:hida}, we obtain
\begin{align}\label{jth:comp:f}\begin{split}
f_\mathbf{j}\bigg(\begpmat y & x \\ 0 & 1 \endpmat\bigg)=&|y|_{\A_F}\sum_{\xi\in F^\times,[\xi]=J}a_f(\xi y\fd_F)W_{\mathbf{k},\mathbf{m}}^\mathbf{j}(\xi y_\infty)\textbf{e}_F(\xi x)
\end{split}\end{align}
where
$$
W_{\mathbf{k},\mathbf{m}}^\mathbf{j}(y_\infty):=\prod_{\s\in\Sigma(\R)}W_{\mathbf{k},\mathbf{m}}^{\s}(y_\s)\cdot\prod_{\s\in\Sigma(\C)}W_{\mathbf{k},\mathbf{m}}^{\s,j_\s}(y_\s)
$$
and $W_{\mathbf{k},\mathbf{m}}^\s(y_\s)(\mathbf{x}_\s)=\sum_{j_\s=0}^{n_\s^*} W_{\mathbf{k},\mathbf{m}}^{\s,j_\s}(y_\s)X_\s^{n_\s^*-j_\s}Y_\s^{j_\s}$ for $\s\in\Sigma(\C)$. Let us set the Haar measure $d^\times y$ on $F^\times\backslash\A_F^\times$ which satisfies the following conditions:
$$
d^\times y_\sigma = \left\{ \begin{array}{ll}
\frac{dy_\sigma}{y_\sigma} & \textrm{if $\s\in\Sigma(\R)$}\\
\frac{drd\phi}{2\pi r} & \textrm{if $\s\in\Sigma(\C)$ and  $y_\sigma=re^{i\phi}\in F_\s^\times$}\\
\end{array} \right.
\text{, } \int_{\O^\times_{F,v}}d^\times y_v=1\text{ for }v\nmid\infty .
$$
Note that we can easily obtain the following equality:
$$
\{\xi\in F^\times: [\xi]=J\}=\{\xi\epsilon: \xi\in P_{F,J},\ \epsilon\in\O_{F,+}^\times\}
$$
where $\O_{F,+}^\times$ is the group of totally positive units of $F$ and $P_{F,J}$ is a representative set $\{\xi\}$ of $\O_F^\times\backslash F^\times$ such that $[\xi]=J$.
Thus by using (\ref{strong:approx}), (\ref{jth:comp:f}), and the above equation, the Mellin transform of $f_{\mathbf{n}^*/2}$ is given by

\begin{align}\begin{split}\label{mellin:transform1}
&\int_{F^\times\backslash\A_F^\times}f_{\textbf{n}^*/2}\bigg(\begin{pmatrix} y & 0 \\ 0 & 1 \end{pmatrix}\bigg)|y|^{s-1}_{\A_F} d^\times y \\
=&\sum_{i=1}^{h_F}\sum_{\epsilon\in\O_{F,+}^\times}\sum_{\xi\in P_{F,J}}\int_{\O_F^\times\backslash U_F}a_f(a_i\xi\epsilon y\fd_F)|a_i\xi\epsilon y|_{\A_F}^{s}W_{\mathbf{k},\mathbf{m}}^{\textbf{n}^*/2}(\xi\epsilon y_\infty)d^\times y \\
\end{split}\end{align} 
for $\mathfrak{R}(s)>1+\max_{\s\in I_F}m_\s$. As the integral in the above equation is invariant under the change of variable $y\mapsto uy$ for $u\in\O_F^\times$,  (\ref{mellin:transform1}) becomes
\begin{align}\begin{split}\label{mellin:transform1'}
&\sum{}^*\int_{\O_F^\times\backslash U_F}a_f(a_i\xi\epsilon uy\fd_F)|a_i\epsilon uy|_{\A_F}^{s}W_{\mathbf{k},\mathbf{m}}^{\textbf{n}^*/2}(\xi\epsilon uy_\infty)d^\times y \\
=&\frac{1}{[\O_F^\times:\O_{F,+}^\times]}\sum_{i=1}^{h_F}\sum_{\xi\in P_{F,J}}\int_{\widehat{\O}_F^\times}a_f(a_i\xi y^{(\infty)}\fd_F)|a_i y^{(\infty)}|_{\A_F^{(\infty)}}^{s}d^\times y^{(\infty)} \\
&\phantom{blaaaaaaaaaaaaaaaank}\times\int_{F_\infty^\times}W_{\mathbf{k},\mathbf{m}}^{\textbf{n}^*/2}(\xi y_\infty)|y_\infty|_{F_\infty}d^\times y_\infty
\end{split}\end{align}
where
$$
\sum{}^*=\frac{1}{[\O_F^\times:\O_{F,+}^\times]}\sum_{u\in\O_F^\times/\O_{F,+}^\times}\sum_{\epsilon\in\O_{F,+}^\times}\sum_{i=1}^{h_F}\sum_{\xi\in P_{F,J}}.
$$
By the following integration formula (Hida \cite[Section 7]{hida1994critical})
\begin{equation*}
\int_0^\infty y^s{K_{j}(ay)}\frac{dy}{y}=2^{s-2}a^{-s}\Gamma\Big(\frac{s+j}{2}\Big)\Gamma\Big(\frac{s-j}{2}\Big)\text{ if }\mathfrak{R}(s\pm j)>0 
\end{equation*}
and our assumption on the weight $(\mathbf{k},\mathbf{m})$, we have 

\begin{align}\begin{split}\label{mellin:transform2}
&\int_{F_{\infty}^\times} W_{\mathbf{k},\mathbf{m}}^{\textbf{n}^*/2} (\xi y_\infty)|y_\infty|^{s}_{F_\infty}d^\times y_\infty \\
=&\prod_{\s\in\Sigma(\R)}\int_{\R^\times} W_{\mathbf{k},\mathbf{m}}^\s(\xi^\s y_\s)|y_\s|^s d^\times y_\s \cdot\prod_{\s\in\Sigma(\C)}\int_{\C^\times} W_{\mathbf{k},\mathbf{m}}^{\s,n_{\s\co}+1}(\xi^\s y_\s)|y_\s|^{2s}d^\times y_\s \\
=&\prod_{\s\in\Sigma(\R)}\frac{2}{|\xi^\s|^{s}}\int_0^\infty e^{-2\pi y}y^{s-m_\s}\frac{dy}{y}\cdot\prod_{\s\in\Sigma(\C)}\frac{1}{|\xi^\s|^{2s}}\int_0^\infty \begin{spmatrix} n_\s^* \\ n_{\s\co}+1 \end{spmatrix} r^{2(s-m_\s)}K_0(4\pi r)\frac{dr}{r} \\
=&\frac{2^{|\Sigma(\R)|}}{|N(\xi)|^s}\prod_{\s\in\Sigma(\C)}\frac{1}{4}\begin{spmatrix} n_\s^* \\ n_{\s\co}+1 \end{spmatrix}\cdot\prod_{\s\in I_F}(2\pi)^{-(s-m_\s)}\Gamma(s-m_\s)
\end{split}\end{align}
for $\xi\in P_{F,J}$ and $\mathfrak{R}(s)>\max_{\s\in I_F}m_\s$. Combining the equations (\ref{mellin:transform1'}) and (\ref{mellin:transform2}), we obtain

\begin{align}\begin{split}\label{mellin:transform}
&\int_{F^\times\backslash\A_F^\times}f_{\textbf{n}^*/2}\bigg(\begin{pmatrix} y & 0 \\ 0 & 1 \end{pmatrix}\bigg)|y|^{s-1}_{\A_F} d^\times y \\
=&\Gamma_{F,\mathbf{k},\mathbf{m}}(s)\sum_{i=1}^{h_F}\sum_{\xi\in P_{F,J}}\frac{1}{N(\xi\fd_F)^s}\int_{\widehat{\O}_F^\times} a_f(a_i\xi y\delta_F)|a_i y|^{s}_{\A_F^{(\infty)}}d^\times y \\
=&\Gamma_{F,\mathbf{k},\mathbf{m}}(s)\sum_{i=1}^{h_F}\sum_{\fa}\frac{a_f(\fa_i\fa\fd_F)}{N(\fa_i\fa\fd_F)^s}=\Gamma_{F,\mathbf{k},\mathbf{m}}(s) L(s,f)
\end{split}\end{align}
for $\mathfrak{R}(s)>\max_{\s\in I_F}(1+m_\s)$, where $\fa$ runs over the set of non-zero principal fractional ideals of $F$ and 
$$
\Gamma_{F,\mathbf{k},\mathbf{m}}(s):=\frac{2^{|\Sigma(\R)|}|D_F|^s}{[\O_F^\times:\O_{F,+}^\times]}\prod_{\s\in\Sigma(\C)}\frac{1}{4}\begin{spmatrix} n_\s^* \\ n_{\s\co}+1 \end{spmatrix}\cdot\prod_{\s\in I_F}(2\pi)^{-(s-m_\s)}\Gamma(s-m_\s) .
$$
Note that the last equality holds due to the Proposition \ref{four:whit:exp:hida} (1).

\subsection{Fricke involution} To find an integral representation of the $L$-function which converges uniformly on the entire complex plane, we need to cut the integral representation (\ref{mellin:transform}) of completed $L$-function of $f$ into two parts by using the Fricke involution, which is similar to the classical case.

From now on, we assume that $f\in S_{(\mathbf{k},\mathbf{m}),J}(\fN,\chi)$ is a newform.
Let us write $\chi=\chi_\infty\chi^{(\infty)}$ where $\chi_\infty$ and $\chi^{(\infty)}$ are the infinite part and the finite part of $\chi$, respectively. For any finite place $v$ of $F$, let $\varpi_v$ be a uniformizer of $\O_{F,v}$.
Let us define the Fricke involution $W_\fN$ for a newform $f\in S_{(\mathbf{k},\mathbf{m}),J}(\fN,\chi)$ by taking the right translation using an element $\begin{spmatrix} 0 & -1 \\ \varpi_\fN & 0 \end{spmatrix}$ of $\GL_2(\A_F^{(\infty)})$:
$$
W_{\fN}f(g)(\mathbf{x}):=N(\fN)^{1-\frac{k}{2}}(\chi|\cdot|_{\A_F}^{k-2})^{-1}(\text{det}(g))f\bigg(g\begpmat 0 & -1 \\ \varpi_\fN & 0 \endpmat \bigg)(\mathbf{x})
$$
where $\chi|\cdot|_{\A_F}^{k-2}$ is a normalization of $\chi$ and $\varpi_\fN:=\prod_{v|\fN}\varpi_v^{\text{ord}_v(\fN)}$.
Then we have the following fact:

\begin{prop}\label{fricke:hecke} 
$W_{\fN}f$ is an element of $S_{(\mathbf{k},\mathbf{m}),J}(\fN,\ovl{\chi})$. Furthermore, it is a Hecke eigenform.
\end{prop}

\begin{proof} See Hida \cite[Section 8]{hida1994critical}. 
\end{proof}

Using the decomposition $\GL_2(\A_F)=\GL_2(F_\infty)\times\GL_2(\A_F^{(\infty)})$, we have
$$
\begin{pmatrix} y & 0 \\ 0 & 1 \end{pmatrix}\begin{pmatrix} 0 & -1 \\ \varpi_\fN & 0 \end{pmatrix}=\begin{pmatrix} 0 & -1 \\ 1 & 0 \end{pmatrix}\cdot\bigg(\begin{pmatrix} 1 & 0 \\ 0 & y_\infty \end{pmatrix}\begin{pmatrix} 0 & 1 \\ -1 & 0 \end{pmatrix},\begin{pmatrix} \varpi_\fN & 0 \\ 0 & y^{(\infty)} \end{pmatrix}\bigg)
$$
where $\GL_2(F)$ acts on $\GL_2(\A_F)$ diagonally. From this matrix identity and the definition of cusp forms, it is easy to obtain that 
\begin{align}\begin{split}\label{fricke:formula}
W_{\fN}f\bigg(\begin{pmatrix} y & 0 \\ 0 & 1 \end{pmatrix}\bigg)(\textbf{x})=&N(\fN)^{1-\frac{k}{2}}\textbf{e}\bigg(\sum_{\s\in J}\frac{k_\s}{4}-\sum_{\s\in\Sigma(\R)\backslash J}\frac{k_\s}{4}\bigg)|y|_{\A_F}^{2-k} \\
&\times f\bigg(\begin{pmatrix} \varpi_\fN y^{-1} & 0 \\ 0 & 1 \end{pmatrix}\bigg)\bigg(\begin{pmatrix} 0 & 1 \\ -1 & 0 \end{pmatrix}\textbf{x}\bigg) .
\end{split}\end{align}
Also from the definition of $W_\fN$ and Proposition \ref{fricke:hecke}, we can check that 
$$
W_\fN^2 f=\chi^{(\infty)}(-1)f=(-1)^{[F:\Q](k-2)}f.
$$
From these formulas, we can obtain analytic continuation and the functional equation of a completed $L$-function $\Lambda(s,f)$ attached to $f$, which is defined by 
$$
\Lambda(s,f):=\Gamma_{F,\mathbf{k},\mathbf{m}}(s)N(\fN)^{\frac{s}{2}}L(s,f).
$$
\begin{thm}\label{Lftn:analytic:conti} 
The completed $L$-function $\Lambda(s,f)$ of $f$ has analytic continuation to $\C$.
Furthermore, if
\begin{equation}\label{parity:cond}
(-1)^{|\Sigma(\R)|(k-2)}C_{F,J,\mathbf{k}}^2=1 
\end{equation}
where $C_{F,J,\mathbf{k}}$ is a number defined by
$$
C_{F,J,\mathbf{k}}=(-1)^{|\Sigma(\R)|+\sum_{\s\in\Sigma(\C)}(n_\s+1)}\mathbf{e}\bigg(\sum_{\s\in\Sigma(\R)\backslash J}\frac{k_\s}{4}-\sum_{\s\in J}\frac{k_\s}{4}\bigg),
$$
then we have following functional equation:
\begin{equation}\label{functional:eq}
\Lambda(s,f)=C_{F,J,\mathbf{k}}\Lambda(k-s,W_\fN f)
\end{equation}
\end{thm}

\begin{rem}\label{parity:cond:rem}
The condition (\ref{parity:cond}) is satisfied, for example, when the weights $k_\s$ are all even.
\end{rem}

\begin{proof}
By splitting integral (\ref{mellin:transform}) into two pieces, applying a change of variable $y\mapsto\varpi_\fN y^{-1}$ and putting (\ref{fricke:formula}) into one of the integrals, one can obtain the following integral representation: 
\begin{align*}
\Lambda(s,f)=&N(\fN)^{\frac{s}{2}}\int_{|y|_{\A_F}\geq |\varpi_\fN|_{\A_F}^{1/2}} f_{\mathbf{n}^*/2}\bigg(\begin{pmatrix} y & 0 \\ 0 & 1 \end{pmatrix}\bigg)|y|_{\A_F}^{s-1}d^\times y \\
&+C_{F,J,\mathbf{k}}N(\fN)^{\frac{k-s}{2}}\int_{|y|_{\A_F}\geq |\varpi_\fN|_{\A_F}^{1/2}} W_{\fN}f_{\mathbf{n}^*/2}\bigg(\begin{pmatrix} y & 0 \\ 0 & 1 \end{pmatrix}\bigg)|y|_{\A_F}^{k-1-s}d^\times y .
\end{align*}
From the above formula, we can easily obtain our functional equation.
\end{proof}

We follow Luo-Ramakrishnan \cite{luo1997determination}. Let $\Phi$ be an infinitely differentiable function
on $\R_{>0}^\times$ with compact support and $\int_0^\infty \Phi(y)\frac{dy}{y}=1$. Define
\begin{align*}
V_{1,s}(x)&:=\frac{1}{2\pi i}\int_{2-i\infty}^{2+i\infty}\kappa(t)\Gamma_{F,\mathbf{k},\mathbf{m}}(s+t) x^{-t}\frac{dt}{t}\\
V_{2,s}(x)&:=\frac{1}{2\pi i}\int_{2-i\infty}^{2+i\infty}\kappa(-t)\Gamma_{F,\mathbf{k},\mathbf{m}}(s+t) x^{-t}\frac{dt}{t}
\end{align*}
where $\kappa(t):=\int_0^\infty\Phi(y)y^t\frac{dy}{y}$. By shifting the contour, one can show that $V_{1,s}$ and $V_{2,s}$ satisfy following:
\begin{align}\begin{split}\label{aux.func.esti}
V_{i,s}(x)&=O(\Gamma_{F,\mathbf{k},\mathbf{m}}(\mathfrak{R}(s)+j))x^{-j})\text{ for all }j\geq 1 \text{ as }x\rightarrow\infty \\
V_{i,s}(x)&=\Gamma_{F,\mathbf{k},\mathbf{m}}(s)+O\bigg(\Gamma_{F,\mathbf{k},\mathbf{m}}\Big(\mathfrak{R}(s)-\frac{1}{2}\Big)x^{\frac{1}{2}}\bigg) \text{ as }x\rightarrow 0
\end{split}\end{align}
where the implicit constants depend only on $j$ and $\Phi$. For $\mathfrak{R}(s)>\frac{k+2}{2},\ \max_{\s\in I_F}(m_\s-2)$ and $y>0$, we have
\begin{align}\begin{split}\label{series:exp}
\frac{1}{2\pi i}\int_{2-i\infty}^{2+i\infty}\kappa(t)\Gamma_{F,\mathbf{k},\mathbf{m}}(s+t) L(s+t,f) y^t\frac{dt}{t}=\sum_{0\neq\fa<\O_F}\frac{a_f(\fa)}{N(\fa)^s}V_{1,s}\bigg(\frac{N(\fa)}{y}\bigg) .
\end{split}\end{align}
By the Residue theorem, we have
\begin{align*}
\Gamma_{F,\mathbf{k},\mathbf{m}}(s) L(s,f)=&\frac{1}{2\pi i}\int_{2-i\infty}^{2+i\infty}\kappa(t)\Gamma_{F,\mathbf{k},\mathbf{m}}(s+t) L(s+t,f)y^t\frac{dt}{t} \\
&+\frac{1}{2\pi i}\int_{-2+i\infty}^{-2-i\infty}\kappa(t)\Gamma_{F,\mathbf{k},\mathbf{m}}(s+t) L(s+t,f)y^t\frac{dt}{t} .
\end{align*}
Putting equations (\ref{functional:eq}) and (\ref{series:exp}) in the above equation, we obtain
\begin{align}\begin{split}\label{approx:func:eq}
\Lambda(s,f)=&N(\fN)^{\frac{s}{2}}\sum_{0\neq\fa<\O_F}\frac{a_f(\fa)}{N(\fa)^s}V_{1,s}\bigg(\frac{N(\fa)}{y}\bigg) \\
&+C_{F,J,\mathbf{k}}N(\fN)^{\frac{k-s}{2}}\sum_{0\neq\fa<\O_F}\frac{a_{W_\fN f}(\fa)}{N(\fa)^{k-s}}V_{2,k-s}\bigg(\frac{N(\fa)y}{N(\fN)}\bigg) 
\end{split}\end{align}
when $(-1)^{|\Sigma(\R)|(k-2)}C_{F,J,\mathbf{k}}^2=1$.

\subsection{Twisted cusp forms}\label{twisted:cuspforms}
We will define and discuss a cusp form twisted by a Hecke character. In this subsection, we follow Hida \cite[Section 6]{hida1994critical}.

Let $\fc$ be an integral ideal of $F$ coprime to $\fd_F$ and $\varphi:F^\times\backslash\A_F^\times\rightarrow\C^\times$ a Hecke character of finite order with conductor $\fc$. Denote by $(\fc^{-1}/\O_F)^\times$ the set of elements of $\prod_{v|\fc}F_{v}$ corresponing to the elements of $\fc^{-1}/\O_F$ whose annihilator ideal is same as $\fc$. For $a\in\A_F^{(\infty)}$, define the Gauss sum $G(\varphi,a)$ of $\varphi$ (Hida \cite[Section 6]{hida1994critical}) by 
$$
G(\varphi,a):=\varphi^{-1}(d_F)\sum_{u\in (\fc^{-1}/\O_F)^\times}\varphi(\varpi_\fc u)\textbf{e}_F(d_F^{-1}au)
$$
which is clearly independent on a choice of $d_F$. Especially, we put $G(\varphi,1^{(\infty)})=G(\varphi)$. Then we have the following lemma for the Gauss sums:

\begin{lem}\label{gauss:sums}
\begin{enumerate}
\item For any $a\in\A_F^{(\infty)}$, we have $G(\varphi)\ovl{\varphi}(a)=G(\varphi,a)$.
\item $|G(\varphi)|^2=N(\fc)$.
\item $\ovl{G(\varphi)}=\varphi^{(\infty)}(-1)G(\ovl{\varphi})$.
\item $G(\ovl{\varphi})G(\varphi)=\varphi^{(\infty)}(-1)N(\fc)$.
\end{enumerate}
\end{lem}

\begin{proof}
Those are immediate consequences of Neukirch \cite[Chapter VII, Proposition 7.5]{neukirch2013algebraic} by using a bijection $(\fc^{-1}/\O_F)^\times\cong(\widehat{\O}_F/\widehat{\fc})^\times$ defined by $u\mapsto \varpi_{\fc}u$.
\end{proof}

Define a function $f\otimes\varphi:\text{GL}_2(\A_F)\rightarrow L(\mathbf{n}^*,\C)$ by
\begin{equation}\label{twisted:cuspform:decomp:eq}
f\otimes\varphi(g)(\mathbf{x}):=G(\varphi)^{-1}\varphi(\det(g))\sum_{u\in(\fc^{-1}/\O_F)^\times}\varphi(\varpi_\fc u)f\bigg(g\begpmat 1 & u \\ 0 & 1 \endpmat\bigg)(\mathbf{x}).
\end{equation}
Then we have the following proposition:

\begin{prop}\label{twisted:cuspform:decomp} We have $f\otimes\varphi\in S_{k}(\fN\cap\fc^2,\chi\varphi^2)$. Also we have $a_{f\otimes\varphi}(\fa)=a_f(\fa)\varphi(\fa)$.
\end{prop}

\begin{proof}
A Hecke character $\varphi$ of finite order with conductor $\fc$ can be considered as a ray class character. Then our proof is immediate by Hida \cite[Section 6]{hida1994critical}.
\end{proof}

Also we have the following relation between the Fricke involution and the twisting by Hecke characters:

\begin{prop}\label{fricke:twist:commute} If $\fc$ and $\fN$ are coprime, then we have
$$
W_{\fN\fc^2}(f\otimes\varphi)=W(\varphi)(W_{\fN}f)\otimes\ovl{\varphi}
$$
where $W(\varphi):=N(\fc)^{1-k}\chi(\varpi_\fc)\ovl{\chi}_\fN\big(\begin{spmatrix} d & -v\varpi_{\fc} \\ -\varpi_{\fN\fc}u & \varpi_{\fc} \end{spmatrix}\big)\varphi(-\varpi_{\fc}^2)G(\ovl{\varphi})^2$, a complex number of absolute value $1$ and $\chi_\fN\big(\begin{spmatrix} a & b \\ c & d \end{spmatrix}\big):=\prod_{v|\fN}\chi(d_v)$.
\end{prop}

\begin{proof} 
In this proof, we keep using the bijection in the proof of Lemma \ref{gauss:sums}. By our assumption, $\varpi_{\fc}$ and $\varpi_{\fN}\varpi_{\fc} u$ are coprime for any $u\in(\fc^{-1}/\O_F)^\times$, thus there exist $d$ and $\varpi_{\fc}v\in\widehat{\O}_F$ such that
$$
\varpi_{\fc}d-\varpi_{\fN}\varpi_{\fc}u\varpi_{\fc}v=1.
$$
From this, we can easily check that the map $\varpi_{\fc}u\mapsto\varpi_{\fc}v$ is well-defined automorphism on $(\widehat{\O}_F/\widehat{\fc})^\times$.
By the definition of the Fricke involution, the equation (\ref{twisted:cuspform:decomp:eq}), and Proposition \ref{twisted:cuspform:decomp}, we have
\begin{align*}
W_{\fN\fc^2}(f\otimes\varphi)(g)(\mathbf{x})=&N(\fN\fc^2)^{1-\frac{k}{2}}(\chi\varphi^2|\cdot|^{k-2})^{-1}(\det(g))G(\varphi)^{-1}\varphi(\det(g))\varphi(\varpi_{\fN\fc^2}) \\
&\times\sum_{u\in(\fc^{-1}/\O_F)^\times}\varphi(\varpi_\fc u)f\bigg(g\begpmat 0 & -1 \\ \varpi_{\fN\fc^2} & 0 \endpmat\begpmat 1 & u \\ 0 & 1 \endpmat\bigg)(\mathbf{x}) .
\end{align*}
Put the identities
$$
\begpmat 0 & -1 \\ \varpi_{\fN\fc^2} & 0 \endpmat\begpmat 1 & u \\ 0 & 1 \endpmat=\begpmat \varpi_{\fc} & 0 \\ 0 & \varpi_{\fc} \endpmat\begpmat 1 & v \\ 0 & 1 \endpmat\begpmat d & -v\varpi_{\fc} \\ -\varpi_{\fN\fc}u & \varpi_{\fc} \endpmat\begpmat 0 & -1 \\ \varpi_{\fN} & 0 \endpmat
$$
and $\varpi_{\fN\fc^2}uv\equiv -1\mod\widehat{\fc}$ into the above equation, 
then by the Definition \ref{adel:cuspform:defn} and Proposition \ref{fricke:hecke}, we obtain
\begin{align*}
W_{\fN\fc^2}(f\otimes\varphi)(g)(\mathbf{x})=&N(\fc)^{2-k}\chi(\varpi_\fc)\ovl{\chi}_\fN\big(\begin{spmatrix} d & -v\varpi_{\fc} \\ -\varpi_{\fN\fc}u & \varpi_{\fc} \end{spmatrix}\big) G(\varphi)^{-1}\ovl{\varphi}(\det(g))\varphi(-\varpi_{\fc}^2) \\
&\times\sum_{u\in(\fc^{-1}/\O_F)^\times}\ovl{\varphi}(\varpi_\fc v)W_\fN f\bigg(g\begpmat 1 & v \\ 0 & 1 \endpmat\bigg)(\mathbf{x}) .
\end{align*}
Using Lemma \ref{gauss:sums}, we can rewrite the above equation as
\begin{align*}
W_{\fN\fc^2}(f\otimes\varphi)=N(\fc)^{1-k}\chi(\varpi_\fc)\ovl{\chi}_\fN\big(\begin{spmatrix} d & -v\varpi_{\fc} \\ -\varpi_{\fN\fc}u & \varpi_{\fc} \end{spmatrix}\big)G(\ovl{\varphi})^2\varphi(-\varpi_{\fc}^2)(W_\fN f)\otimes\ovl{\varphi} ,
\end{align*}
hence we are done.
\end{proof}

\begin{note}\label{prime:power:conductor}
Note that if $\fc$ is a prime power, then one can easily observe that $\varphi(\varpi_\fc)=1$, thus $W(\varphi)=N(\fc)^{1-k}\chi(\varpi_\fc)\ovl{\chi}_\fN\big(\begin{spmatrix} d & -v\varpi_{\fc} \\ -\varpi_{\fN\fc}u & \varpi_{\fc} \end{spmatrix}\big)\varphi(-1)G(\ovl{\varphi})^2$ if $\fc$ is a prime power.
\end{note}

Finally, we can find a fast convergent series expression of the twisted special $L$-values:
Suppose that $\fc$ and $\fd_F\fN$ are coprime, and $(-1)^{|\Sigma(\R)|(k-2)}C_{F,J,\mathbf{k}}^2=1$. Then by the equation (\ref{approx:func:eq}) and Proposition \ref{twisted:cuspform:decomp}, \ref{fricke:twist:commute}, we have
\begin{align}\label{fast:int:exp}
\begin{split}
\Gamma_{F,\mathbf{k},\mathbf{m}}\Big(\frac{k}{2}\Big) L\Big(\frac{k}{2},f\otimes\varphi\Big)
=&\sum_{0\neq\fa<\O_F}\frac{a_f(\fa)\varphi(\fa)}{N(\fa)^{k/2}}V_{1,\frac{k}{2}}\bigg(\frac{N(\fa)}{y}\bigg) \\
&+C_{F,J,\mathbf{k}}W(\varphi)\sum_{0\neq\fa<\O_F}\frac{a_{W_\fN f}(\fa)\ovl{\varphi}(\fa)}{N(\fa)^{k/2}}V_{2,\frac{k}{2}}\bigg(\frac{N(\fa)y}{N(\fN\fc^2)}\bigg) .
\end{split}\end{align}

\section{Galois Averages of Hecke characters}\label{hecke:char}
In this section, we are going to discuss the Galois averages of Hecke characters, which play a crucial role to show the non-vanishing of the special $L$-values. 

Let $K/\Q$ be a finite extension and set $n_0:=\textrm{max}\{m\in\Z|\ \mu_{m}\subset K\}$.
For a Hecke character $\varphi:F^\times\backslash\A_F^\times\rightarrow\C^{\times}$ of finite order with conductor $\fc$, or a ray class character, we define the Galois averages of $\varphi$ over $K$ by 
\begin{align*}
&\varphi_{\textrm{av}}:=\frac{1}{[K(\varphi):K]}\sum_{\sigma\in\textrm{Gal}(K(\varphi)/K)}\varphi^\sigma
\\
&\varphi^\iota_{\textrm{av}}:=\frac{1}{[K(\varphi):K]}\sum_{\sigma\in\textrm{Gal}(K(\varphi)/K)}W(\varphi^\sigma)\ovl{\varphi}^\sigma
\end{align*}
where $K(\varphi)$ is the field determined by adjoining all the values of $\varphi$ to $K$.

Let $\fp$ be a prime ideal of $F$ lying above $p$. We assume that $\fp$ and $\fd_F\fN$ are coprime, which allows us to use the discussion in Subsection \ref{twisted:cuspforms}. We assume that $p$ is coprime to $h_F$, and $\fa_i$ is coprime to $\fp$ for all $i=1,\cdots,h_F$. Let us fix an embedding $F\hookrightarrow F_\fp$. Let $\text{Cl}(F,\mathfrak{m})$ be the ray class group of a modulus $\mathfrak{m}$ of $F$.
Let us denote 
$$
\O_{F,\fp}:=\varprojlim_m \O_F/\fp^m,\ \ovl{\O_F^\times}:=\varprojlim_m \O_F^\times\text{ mod }\fp^m,\text{ and }\text{Cl}(F,\fp^\infty):=\varprojlim_m \text{Cl}(F,\fp^m).
$$
Then we have the following exact sequences:
\begin{equation}\label{ray:class:group:seq}
\begin{tikzcd} 1 \arrow{r} & \overline{\O_{F}^\times} \arrow{r} & \O_{F,\fp}^\times \arrow{r} & \text{Cl}(F,\fp^\infty) \arrow{r} & \text{Cl}(F) \arrow{r} & 1 \end{tikzcd} ,
\end{equation}
$$
\begin{tikzcd} 1 \arrow{r} & 1+\fp\O_{F,\fp} \arrow{r} & \O_{F,\fp}^\times \arrow{r}{\text{mod }\fp} & (\O_{F,\fp}/\fp)^\times \arrow{r} & 1  \end{tikzcd} .
$$
Let $\Delta$ be the torsion part of $\text{Cl}(F,\fp^\infty)$, $W$ be a split image of $(\O_{F,\fp}/\fp)^\times$ in $\O_{F,\fp}^\times$, and $\Gamma^\p:=1+\fp\O_{F,\fp}$. By decomposing (\ref{ray:class:group:seq}) into the torsion part and the pro $p$-part, then we obtain the following exact sequences
$$
\begin{tikzcd} W \arrow{r} & \Delta \arrow{r} & \text{Cl}(F) \arrow{r} & 1  \end{tikzcd} ,
\begin{tikzcd} \Gamma^\p \arrow{r} & \text{Cl}(F,\fp^\infty)_p \arrow{r} & 1  \end{tikzcd} .
$$

Let us denote $\mu_{\infty}:=\varinjlim_{n}\mu_n$ and $\Xi_\fp:=\Hom_{\text{cont}}(\text{Cl}(F,\fp^\infty),\mu_{\infty})$. Then there is a unique element $(\psi_n)_n\in\varinjlim_{n}\Hom_{\text{cont}}(\text{Cl}(F,\fp^n),\mu_{\infty})\cong\Xi_\fp$ corresponding to $\psi$. From now on, let us say that $\psi\in\Xi_\fp$ is primitive if and the conductor of $\psi_n$ is $\fp^{n+n_0}$ for all $n$. Let $\psi=(\psi_n)_n$ be a primitive element of $\Xi_\fp$, and $\widetilde{\psi}$  be a lifting of $\psi$ to $\Gamma^\p$. Let $\mathcal{E}:=\ker(\widetilde{\psi})$, then we have the following split exact sequence
$$
\begin{tikzcd} 1 \arrow{r} & \mathcal{E} \arrow{r}{\subset} & \Gamma^\p \arrow{r}{\widetilde{\psi}} & \mu_{\infty} \arrow{r} & 1\end{tikzcd} .
$$

Let $\Gamma$ be a split image of $\mu_{\infty}$ in $\Gamma^\p$, whose $\Z_p$-rank is one.
From the above exact sequences, we obtain the surjections
\begin{equation}\label{rayclassgroup:surjection}
\begin{tikzcd} \Delta\times\Gamma^\p\cong\Delta\times\mathcal{E}\times\Gamma \arrow[r, two heads] & \text{Cl}(F,\fp^\infty) \arrow[r, two heads] & \text{Cl}(F,\fp^n) \end{tikzcd} ,
\end{equation}
hence $\psi_n$ can be considered as an element of $\Hom_{\text{cont}}(\Delta\times\mathcal{E}\times\Gamma,\mu_{\infty})$ for each $n$.

Let us define a filtration on $\Gamma^\p$ by $\Gamma^\p_n:=1+\fp^n\O_F$. Let us denote $\mathcal{E}_n:=\mathcal{E}\cap\Gamma^\p_n$ and $\Gamma_n:=\Gamma\cap\Gamma^\p_n$. As $\psi_n$ has a $p$-power order and $(p,h_F)=1$, we have $\psi_n(\Delta)=\{1\}$. So we have $\ker(\psi_n)\supset \Delta\times\mathcal{E}\times\Gamma_{n+n_0}$. The following proposition tells us that nonzero integral elements which make the Galois averages non-zero, are distributed sparsely.

\begin{prop}
\label{gal:av:nonzero} Let $(h_F,p)=1$. Let $\fa$ be an element of $\operatorname{Cl}(F,\fp^\infty)$ and $\psi=(\psi_n)_n$ a primitive element of $\Xi_\fp$. Then for $n\geq 1$, we have that $\psi_{n,\operatorname{av}}(\fa)\neq 0$ if and only if 
$$
\fa\in\bigcup_{(\fb,\epsilon,\gamma)\in\Delta\times\mathcal{E}\times\Gamma_n} \{\fb\epsilon\gamma\}=\bigcup_{(\fb,\epsilon,\gamma)\in\Delta\times(\mathcal{E}/\mathcal{E}_n)\times\Gamma_n^\p}\{\fb\epsilon\gamma\}.
$$
\end{prop}

\begin{proof}
Assume that $\psi_{n,\textrm{av}}(\fa)\neq 0$, then 
we have $\psi_n(\fa)\in\mu_{\infty}$ by our assumption. As the conductor of $\psi_n$ is $\fp^{n+n_0}$, $\psi$ factors through $\Gamma/\Gamma_{n}\cong\mu_{{n-1}}$, hence $\psi_n(\fa)=\zeta_{{n-1}}^r$ for some $r\in\Z_{\geq0}$.
Our assumption $\psi_{n,\textrm{av}}(\fa)\neq 0$ also says that 
$$
\Tr_{K(\zeta_{{n-1}}^r)/K}(\zeta_{{n-1}}^r)\neq 0 ,
$$ 
thus we have $n-1-v_p(r)\leq n_0$. Note that we can write $\fa=\fb\epsilon\gamma$ for some $(\fb,\epsilon,\gamma)\in\Delta\times\mathcal{E}\times\Gamma$, thus we have
$$
\psi_n(\fa)=\psi_n(\gamma)=\zeta_{{n-1}}^r\in\mu_{{n_0}},
$$
which implies that $\gamma\in\Gamma_{n}/\Gamma_{n+n_0}\cong\mu_{n_0}$.

Conversely, if $\fa=\fb\epsilon\gamma$ for some $(\fb,\epsilon,\gamma)\in\Delta\times\mathcal{E}\times\Gamma_n$, then $\psi_n(\fa)=\psi_n(\gamma)\in\mu_{{n_0}}$. Thus, $\psi_n(\fa)=\zeta_{{n_0}}^r$ for some $r\in\Z_{\geq0}$. Hence the Galois average of $\psi_n(\fa)$ is given by
$$
\psi_{n,\text{av}}(\fa)=\frac{1}{[K(\psi_n):K]}\sum_{\s\in\textrm{Gal}(K(\psi_n)/K)}(\zeta_{{n_0}}^r)^{\s}=\zeta_{{n_0}}^r\neq 0 .
$$

Assume that $\fa=\fb\epsilon\g$ for some $\g\in \Gamma^\p$. Then by (\ref{rayclassgroup:surjection}), we have $\fb\epsilon\g=\fb^\p\epsilon^\p\g^\p$ for some $\fb^\p\epsilon^\p\g^\p\in \Delta\times\mathcal{E}\times\Gamma$. Then $\psi_n(\gamma)=\psi_n(\gamma^\p)$, which implies that $\g\in\gamma^\p\Gamma=\Gamma$. The converse direction is clear as $\Gamma\subset\Gamma^\p$. In conclude, $\g\in\Gamma_{n}$.

For $\epsilon_1,\epsilon_2\in\mathcal{E}$, we can easily check that $\epsilon_1\Gamma^\p_n=\epsilon_2\Gamma^\p_n$ if and only if $\epsilon_1\epsilon_2^{-1}\in\mathcal{E}_n$.
\end{proof}

The following lemma allow us to estimate the size of $\psi_{n,\text{av}}^\iota$.

\begin{prop}\label{root:num:gal:av} For any integral ideal $\fa$ of $F$, We have 
$$
|\psi^\iota_{n,\operatorname{av}}(\fa)|\ll_{F,n_0,\fp} N(\fp)^{n(\frac{1}{2}-\frac{1}{\operatorname{f}(F,\fp)})}
$$
where $\operatorname{f}(F,\fp)$ is the residue degree of $F$ at $\fp$.
\end{prop}

\begin{proof}
If $\fa$ is not coprime to $\fp$, then clearly the above inequality holds. Thus $[\fa]_n\in\text{Cl}(F,\fp^n)$. Then $\fa=\fa_i\a$ for some $i$ and $\a\in\O_F$ which is coprime to $\fp$. By the definition of $W(\varphi)$ (see Proposition \ref{fricke:twist:commute} and Note \ref{prime:power:conductor}), we have
\begin{align}\begin{split}\label{rootnum:char:pre}
|\psi^\iota_{n,\operatorname{av}}(\fa)|
=\frac{1}{N(\fp)^{n+n_0}}\bigg|\frac{1}{|G|}\sum_{\s\in G}\ovl{\psi}^\s_{n}(\fa_i\a\beta d_F^{-2})G(\ovl{\psi}_n^\s)^2\bigg|
\end{split}\end{align}
where $G=\text{Gal}(K(\psi_n)/K)$. By the definition of Gauss sum, we have
\begin{align}\begin{split}\label{square:Gauss:sum}
G(\ovl{\psi}_n^\s)^2=\sum_{u,v}\ovl{\psi}_n^\s(uv)\textbf{e}_F\bigg(\frac{u+v}{d_F\varpi_\fp^{n+n_0}}\bigg)
\end{split}\end{align}
where $\beta,u$ and $v$ runs over $(\O_F/\fp^{n+n_0})^\times$. 
Combining the equations (\ref{rootnum:char:pre}) and (\ref{square:Gauss:sum}), and using change of variable $\beta=uv$, we obtain
\begin{align}\begin{split}\label{rootnum:char}
|\psi^\iota_{n,\operatorname{av}}(\fa)|
=\frac{1}{N(\fp)^{n+n_0}}\bigg|\sum_{\beta}\ovl{\psi}_{n,\operatorname{av}}(\fa_i\a\beta d_F^{-2})\sum_{uv=\beta}\textbf{e}_F\bigg(\frac{u+v}{d_F\varpi_\fp^{n+n_0}}\bigg)\bigg|.
\end{split}\end{align}
As $\a d_F^{-2}$ is coprime to $\fp$, we can abbreviate $\a\beta d_F^{-2}$ by $\beta$.
By Proposition \ref{gal:av:nonzero},
$\psi_{n,\textrm{av}}(\beta)\neq 0$ if and only if $\beta\in W\cdot\mathcal{E}/\mathcal{E}_{n+n_0}\cdot\Gamma_n/\Gamma_{n+n_0}$. 
So (\ref{rootnum:char}) is equal to
\begin{align}\begin{split}\label{rootnum:char:2}
\frac{1}{N(\fp)^{n+n_0}}\bigg|\sum_{\k\in W}\sum_{\epsilon\in\mathcal{E}/\mathcal{E}_{n+n_0}}\sum_{\gamma\in\Gamma_n/\Gamma_{n+n_0}}\ovl{\psi}_{n,\text{av}}(\gamma)\sum_{u}\textbf{e}_F\bigg(\frac{u+\k\epsilon\gamma u^{-1}}{d_F\varpi_\fp^{n+n_0}}\bigg)\bigg| .
\end{split}\end{align}
By Bruggeman-Miatello\cite[Proposition 9]{bruggeman1995estimates}, which is about the estimation on the Kloosterman sums, (\ref{rootnum:char:2}) is less than
\begin{align*}
&\frac{1}{N(\fp)^{n+n_0}}\sum_{\k\in W}\sum_{\epsilon\in\mathcal{E}/\mathcal{E}_{n+n_0}}\sum_{\gamma\in\Gamma_n/\Gamma_{n+n_0}}\bigg|\sum_{u}\textbf{e}_F\bigg(\frac{u+\k\epsilon\gamma u^{-1}}{d_F\varpi_\fp^{n+n_0}}\bigg)\bigg| \\
\leq&\frac{2}{N(\fp)^{\frac{n+n_0}{2}}}\sum_{\k\in W}\sum_{\epsilon\in\mathcal{E}/\mathcal{E}_{n+n_0}}\sum_{\gamma\in\Gamma_n/\Gamma_{n+n_0}} 1 = \frac{2|W|p^{n_0\cdot\operatorname{rk}_{\Z_p}(\Gamma)}p^{(n+n_0-1)\cdot\operatorname{rk}_{\Z_p}(\mathcal{E})}}{N(\fp)^{\frac{n+n_0}{2}}} .
\end{align*}
Note that we can observe that $\text{rk}_{\Z_p}(\mathcal{E})=\text{rk}_{\Z_p}(\Gamma^\p)-\text{rk}_{\Z_p}(\mu_{\infty})=\text{f}(F,\fp)-1$ and $\text{rk}_{\Z_p}(\Gamma)=\text{rk}_{\Z_p}(\mu_{\infty})=1$, where $\text{rk}_{\Z_p}(M)$ is the $\Z_p$-rank of $M$. Thus we can conclude the proof.
\end{proof}

\section{Number of elements in arithmetic progressions}\label{esti:arith:progress}

In this section, we obtain an estimation of the number of elements and the absolute norms of elements in arithmetic progressions by using the idea of Rohrlich \cite{rohrlich1989nonvanishing}, which plays a key role to estimate our twisted special $L$-values. 

For each $\fa=[\fa_n]_n\in\text{Cl}(F,\fp^\infty)$, let $\fa_n$ be an integral ideal of $F$ which is a representative of a ray class $[\fa_n]\in\text{Cl}(F,\fp^n)$ modulo $\fp^n$. Similarly, for each $\a=(\a_n)_n\in\Gamma^\p$, define $\langle\a\rangle_n\in\O_F\backslash\{0\}$ by a representative of $\a_n\in\Gamma^\p/\Gamma^\p_n$. Let us denote $F^{\fp,n}$ by the elements $\xi$ of $F$ coprime to $\fp$ and $v_p(\xi-1)\geq n$ in $\O_{F,\fp}$. Let us denote by $F_\R:=F\otimes_\Q \R$ the real Minkowski space of $F$, and $C_F^0$ a fundamental domain of $F_\R/\O_F^\times$, where $\O_F^\times$ acts on $F_\R$ via the natural embedding $F\hookrightarrow F_\R$. 

For simplicity, let us assume that $\fp$ is totally split, then $\mathcal{E}$ is trivial.
Then by Proposition \ref{gal:av:nonzero}, we have the following lemma:
\begin{lem}
\label{gal:av:nonzero:integral} Let $(h_F,p)=1$. Let $\fa$ be an integral ideal of $F$ and $(\psi_n)_n\in\Xi_\fp$ a primitive element. Then for $n\geq 1$, we have that $\psi_{n,\operatorname{av}}(\fa)\neq 0$ if and only if 
$$
\fa\in\coprod_{\kappa\in W}\coprod_{i=1}^{h_F}\coprod_{\gamma\in (F^{\fp,n}\cap\fa_{i}^{-1})\cap C_F^0}\{\fa_{i}\langle\kappa\rangle_n\gamma\}.
$$
\end{lem}

\begin{proof}
By Proposition \ref{gal:av:nonzero}, $\psi_{n,\operatorname{av}}(\fa)\neq 0$ if and only if 
$$
\fa\in\coprod_{\fb\in\Delta}\coprod_{\gamma\in (F^{\fp,n}\cap\fa_{i}^{-1})\cap C_F^0}\{\fb_n\gamma\}.
$$
Note that by the definition of $\Delta$, we can write $\fb_n=\fa_{i(\fb)}\langle\kappa_{i(\fb)}\rangle_n$ for some $i(\fb)=1,\cdots,h_F$ and $\kappa_{i(\fb)}\in W$. Thus we have
\begin{align*}
\coprod_{\fb\in\Delta}\coprod_{\gamma\in (F^{\fp,n}\cap\fa_{i}^{-1})\cap C_F^0}\{\fb_n\gamma\}=
\coprod_{\kappa\in W}\coprod_{i=1}^{h_F} \coprod_{\gamma\in (F^{\fp,n}\cap\fa_{i}^{-1}\langle\kappa\rangle_n^{-1})\cap C_F^0}\{\fa_{i}\langle\kappa\rangle_n\gamma\}.
\end{align*}
By the definition of $F^{\fp,n}$ and $W$, we have
\begin{align*}
\coprod_{\kappa\in W}\coprod_{i=1}^{h_F} \coprod_{\gamma\in (F^{\fp,n}\cap\fa_{i}^{-1}\langle\kappa\rangle_n^{-1})\cap C_F^0}\{\fa_{i}\langle\kappa\rangle_n\gamma\}.=\coprod_{\kappa\in W}\coprod_{i=1}^{h_F}\coprod_{\gamma\in (F^{\fp,n}\cap\fa_{i}^{-1})\cap C_F^0} \{\fa_{i}\langle\kappa\rangle_n\gamma\}.
\end{align*}
So we conclude the proof.
\end{proof}

For $n\in\Z_{>0}$, $x>0$, and an integral ideal $\fc$ of $F$ which is coprime to $\fp$, let us define a number $U_{n,\fc}(x)$ by
$$
U_{n,\fc}(x):=\#\{\beta\in (F^{\fp,n}\cap\fc)\cap C_F^0: |N(\beta)|\leq x\}.
$$ 

Let us recall the following fact about coherent cone decomposition. By Rohrlich \cite[Proposition 5]{rohrlich1989nonvanishing}, there is a finite collection $\mathscr{B}$ of coherent $\Z$-cones in $F$ such that 
\begin{equation}\label{coherent:cone:decomp}
\O_F\backslash\{0\}\subset\bigcup_{u\in\O_F^\times}\bigcup_{B\in\mathscr{B}} uB .
\end{equation}
Then we have the following estimations:

\begin{prop}\label{lattice:esti:1} 
For $n\in\Z_{>0}$, $x>0$ and $\a\in\O_F$ coprime to $\fp$, we have 
$$U_{n,\fc}(x)\ll_{F} \operatorname{max}\Big(\frac{x}{N(\fp)^n},1\Big).$$
\end{prop}

\begin{proof} 
The inequality holds by the equation (21) in Rohrlich \cite[Proposition 5]{rohrlich1989nonvanishing} since 
$$
(K^{\fp,n}\cap\fc)\cap C_F^0 \subset(\O_F\backslash\{0\})/\O_F^\times\subset\bigg(\bigcup_{u\in\O_F^\times}\bigcup_{B\in\mathscr{B}}uB\bigg) /\O_F^\times=\bigcup_{B\in\mathscr{B}} B 
$$ 
by the equation (\ref{coherent:cone:decomp}).
Also we can observe that the implicit constant in the equation (21) in Rohrlich \cite[Proposition 5]{rohrlich1989nonvanishing} depends only on the coherent cone decomposition $\mathscr{B}$ of $\O_F$, which is a finite collection of cones and depends only on $F$.
\end{proof}

\begin{lem}\label{lattice:esti:2} 
If $\a\in (1+\fp^{n})\backslash\{1\}$, then we have $|N(\a)|\gg_F N(\fp)^{n}$.
\end{lem}

\begin{proof} Let $\a=1+\beta\in (1+\fp^{n})\backslash\{1\}\cap C_F^0$. Then by the equation (\ref{coherent:cone:decomp}), we have 
$$
N(\fp)^n\leq |N(\beta)|=|N(\a-1)|=|N(\a)|\prod_{\s\in I_F}|1-\s(\a)^{-1}| .
$$
Note that for $B\in\mathscr{B}$, there is a basis $\{z_j\}_{j=1}^{d}$ of $F$ over $\Q$ such that $B=\Z_{>0}\langle\{z_j\}_{j=1}^{d}\rangle$ where $d=[F:\Q]$. Let $\s\in I_F$ and $z=\sum_{j=1}^d n_j z_j\in B$, then we have $\max(\{n_j\}_{j=1}^d)\ll_B |\s(z)|$ as $B$ is a coherent $\Z$-cone (Rohrlich \cite{rohrlich1989nonvanishing}). Thus the above equation becomes
$$
N(\fp)^n\leq |N(\a)|\prod_{\s\in I_F} (1+|\s(\a)|^{-1})\ll_F |N(\a)|.
$$
as $\a\in (1+\beta\in (1+\fp^{n})\backslash\{1\}\cap C_F^0\subset\bigcup_{B\in\mathscr{B}}B$ by the equation (\ref{coherent:cone:decomp}). For $\a\in (1+\fp^{n})\backslash\{1\}$, there exists $\a^\p\in(1+\fp^{n})\backslash\{1\}\cap C_F^0$ and $u\in\O_F^\times$ such that $\a=\a^\p u$. Thus we have $N(\a)=N(\a^\p)\gg_F N(\fp)^{n}$ by the above inequality. Hence we are done.
\end{proof}

We have the following lemma:
\begin{lem}\label{lattice:esti:3} 
For $\kappa\in W\backslash\{1\}$, we have $N(\langle\kappa\rangle_n)\gg_F N(\fp)^{n/|W|}$.
\end{lem}

\begin{proof}
By the definition of $W$, we have $\langle\kappa\rangle_n^{|W|}\in (1+\fp^n)\backslash\{1\}$ for $\kappa\in W\backslash\{1\}$. By Lemma \ref{lattice:esti:2}, we have $N(\langle\kappa\rangle_n^{|W|})\gg_F N(\fp)^{n}$. So we can conclude the proof.
\end{proof}

\section{Galois averages of the special $L$-values}\label{estimation}

In this section, we obtain an estimation on the twisted special $L$-values which allows us to verify the non-vanishing of $L$-values under the assumption that $\fp$ splits completely over $\Q$. 

Let us recall that $\fp$ is a prime ideal of $F$ lying above $p$, and coprime to $h_F\fd_F\fN$. For $x>0$ and an integral ideal $\fc$ which is coprime to $\fp$, we set $U_{n,\fc}(x):=\#\{\beta\in (F^{\fp,n}\cap\fc)\cap C_F^0: |N(\beta)|\leq x\}$, and $C_F^0=F_\R/\O_F^\times$.

Let $f\in S_{(\mathbf{k},\mathbf{m}),J}(\fN,\chi)$ be a newform and $K_f$ the Hecke field of $f$ over $\Q$ (cf. subsection \ref{Main:Theorems}). Let $(-1)^{|\Sigma(\R)|(k-2)}C_{F,J,\mathbf{k}}^2=1$ and set $K=K_f$. 
For a Hecke character $\varphi:F^\times\backslash\A_F^\times\rightarrow\C^\times$ of finite order with conductor $\fc$, or a ray class character, define the Galois average of the twisted special $L$-value by
$$
L_{\textrm{av}}(f\otimes\varphi):=\frac{1}{[K_f(\varphi):K_f]}\sum_{\s\in\text{Gal}(K_f(\varphi)/K_f)} L\Big(\frac{k}{2},f\otimes\varphi^\s\Big) .
$$
Then we can obtain the following estimation on the averaged special $L$-values, which will be proved in the end of this section:
\begin{thm}\label{main:thm:1} Let $\Delta$ be the set described in Section \ref{hecke:char} and $(\psi_n)_n$ be a primitive element of $\Xi_\fp$ where $\Xi_{\fp}=\Hom_{\text{cont}}(\text{Cl}(F,\fp^\infty),\mu_{\infty})$, thus the conductor of $\psi_n$ is $\fp^{n+n_0}$ for each $n$. Assume that $\operatorname{f}(F,\fp)=1$. For $a>1+|W|^{-1}$ and $\e>0$, an estimation of $L_{\operatorname{av}}(f\otimes\psi_n)$ is given by
\begin{align}\begin{split}\label{main:thm:1:eq}
&L_{\operatorname{av}}(f\otimes\psi_n)-\frac{V_{1,\frac{k}{2}}\big(\frac{1}{y}\big)}{\Gamma_{F,\mathbf{k},\mathbf{m}}\big(\frac{k}{2}\big)} \ll_{\e,\fp} N(\fp)^{n(\theta+\e-\frac{1}{2})} \\
&+N(\fp)^{n\big(a(\frac{1}{2}+\theta+\e)-1\big)}+N(\fp)^{n\big((2\theta+2\e+\frac{1}{2})-a(\theta+\e+\frac{1}{2})\big)}
\end{split}\end{align}
as $n$ tends to the infinity.
\end{thm}

By (\ref{fast:int:exp}), the Galois average $L_{\textrm{av}}(f\otimes\psi_n)$ is given by

\begin{align}
&\frac{1}{\Gamma_{F,\mathbf{k},\mathbf{m}}\big(\frac{k}{2}\big)}\sum_{0\neq\fa<\O_F}\frac{a_f(\fa)\psi_{n,\text{av}}(\fa)}{N(\fa)^{k/2}}V_{1,\frac{k}{2}}\bigg(\frac{N(\fa)}{y}\bigg) \label{1st:Lvalue} \\
+&\frac{C_{F,J,\mathbf{k}}}{\Gamma_{F,\mathbf{k},\mathbf{m}}\big(\frac{k}{2}\big)}\sum_{0\neq\fa<\O_F}\frac{a_{W_\fN f}(\fa)\psi^\iota_{n,\text{av}}(\fa)}{N(\fa)^{k/2}}V_{2,\frac{k}{2}}\bigg(\frac{N(\fa)y}{N(\fN\fp^{2n+2n_0})}\bigg) . \label{2nd:Lvalue}
\end{align} 
as $f\otimes\varphi\in S_{(\mathbf{k},\mathbf{m}),J}(\fN\cap\fc^2,\chi\varphi^2)$ by Proposition \ref{twisted:cuspform:decomp}.

Now we will estimate the quantity $L_{\textrm{av}}(f\otimes\psi_n)$ for each $n$. To do this, we need a bound for the Hecke eigenvalues of $f$.
Let us assume that the coefficients satisfy 
$$
|a_f(\fp)|\leq 2N(\fp)^{\frac{k-1}{2}+\theta}
$$
for any prime ideals $\fp$ of $\O_F$ and  a number $\theta\in [0,\frac{1}{2}]$ to be specified in Section \ref{det}. Hence, for $\e>0$ and each integral ideal $\fa$ of $\O_F$, we have 
\begin{equation}\label{rama:peter:bdd}
|a_f(\fa)|\leq 2d(\fa)N(\fa)^{\frac{k-1}{2}+\theta}\ll_\e N(\fa)^{\frac{k-1}{2}+\theta+\e}
\end{equation}
where $d(\fa)$ is the number of the integral ideals of $F$ dividing $\fa$.

From now on, we do not consider the variables related to $f$ and $F$ in the implicit constants of our estimations. First, we estimate the last term (\ref{2nd:Lvalue}) of the averaged $L$-value by following:
\begin{prop}\label{1st:gal:av:Lvalue:error:2} For $\e>0$ and $y>0$, we have that
\begin{align*}
(\ref{2nd:Lvalue})
\ll_{\e,\fp} & \frac{N(\fp)^{n(2\theta+2\e+\frac{3}{2}-\frac{1}{\operatorname{f}(F,\fp)})}}{y^{\theta+\e+\frac{1}{2}}}
\end{align*}
where $\operatorname{f}(F,\fp)$ is the residue degree of $F/\Q$ at $\fp$.
\end{prop}

\begin{proof}
Note that $W_\fN f$ is also an Hecke eigenform by Proposition \ref{fricke:hecke}. By the Ramanujan-Petersson bound (\ref{rama:peter:bdd}) and Proposition \ref{root:num:gal:av}, we obtain
\begin{align}\begin{split}\label{1st:gal:av:Lvalue:error:2:esti}
&\sum_{0\neq\fa<\O_F}\frac{a_{W_\fN f}(\fa)\psi^\iota_{n,\operatorname{av}}(\fa)}{N(\fa)^{k/2}}V_{2,\frac{k}{2}}\bigg(\frac{N(\fa)y}{N(\fN\fp^{2n+2n_0})}\bigg) \\
\ll_{\e,\fp} &N(\fp)^{n(\frac{1}{2}-\frac{1}{\operatorname{f}(F,\fp)})}\sum_{0\neq\fa<\O_F} N(\fa)^{\theta+\e-\frac{1}{2}}V_{2,\frac{k}{2}}\bigg(\frac{N(\fa)y}{N(\fN\fp^{2n+2n_0})}\bigg) \\
=&N(\fp)^{n(\frac{1}{2}-\frac{1}{\operatorname{f}(F,\fp)})}\sum_{m\in\Z_{>0}}\sum_{N(\fa)=m}m^{\theta+\e-\frac{1}{2}}V_{2,\frac{k}{2}}\bigg(\frac{my}{N(\fN\fp^{2n+2n_0})}\bigg) \\
\ll_\e &N(\fp)^{n(\frac{1}{2}-\frac{1}{\operatorname{f}(F,\fp)})}\sum_{m>0} m^{\theta+\e-\frac{1}{2}}V_{2,\frac{k}{2}}\bigg(\frac{my}{N(\fN\fp^{2n+2n_0})}\bigg) 
\end{split}\end{align}
where in the last inequality, we use the fact that $\sum_{N(\fa)=m}1\leq d(m)\ll_\e m^\e$.
We can split the last term of the above equation into two parts:
\begin{align*}
(\ref{1st:gal:av:Lvalue:error:2:esti})=I+II,\text{ where }I=\sum_{m>\frac{N(\fN\fp^{2n+2n_0})}{y}},\ II=\sum_{0<m<\frac{N(\fN\fp^{2n+2n_0})}{y}}
\end{align*}
Using (\ref{aux.func.esti}), for $j\geq 1$, we have
\begin{align*}\begin{split}
I&\ll N(\fp)^{n(\frac{1}{2}-\frac{1}{\operatorname{f}(F,\fp)})}\sum_{m>\frac{N(\fN\fp^{2n+2n_0})}{y}} m^{\theta+\e-\frac{1}{2}}\bigg(\frac{m y}{N(\fN\fp^{2n+2n_0})}\bigg)^{-j} \\
&\ll N(\fp)^{n(\frac{1}{2}-\frac{1}{\operatorname{f}(F,\fp)})}\frac{N(\fN\fp^{2n+2n_0})^j}{y^j}\int_{\frac{N(\fN\fp^{2n+2n_0})}{y}}^\infty x^{\theta+\e-\frac{1}{2}-j} dx \\
&\ll_\e N(\fp)^{n(\frac{1}{2}-\frac{1}{\operatorname{f}(F,\fp)})}\frac{N(\fN\fp^{2n+2n_0})^{\theta+\e+\frac{1}{2}}}{y^{\theta+\e+\frac{1}{2}}} \ll \frac{N(\fp)^{n(2\theta+2\e+\frac{3}{2}-\frac{1}{\operatorname{f}(F,\fp)})}}{y^{\theta+\e+\frac{1}{2}}} .
\end{split}\end{align*} 
Similarly, we have
\begin{align*}\begin{split}
II&\ll N(\fp)^{n(\frac{1}{2}-\frac{1}{\operatorname{f}(F,\fp)})}\sum_{0<m<\frac{N(\fN\fp^{2n+2n_0})}{y}} m^{\theta+\e-\frac{1}{2}} \\ 
&\ll N(\fp)^{n(\frac{1}{2}-\frac{1}{\operatorname{f}(F,\fp)})}\int_0^{\frac{N(\fN\fp^{2n+2n_0})}{y}} x^{\theta+\e-\frac{1}{2}}dx \ll_\e\frac{N(\fp)^{n(2\theta+2\e+\frac{3}{2}-\frac{1}{\operatorname{f}(F,\fp)})}}{y^{\theta+\e+\frac{1}{2}}} .
\end{split}\end{align*} 
Hence we can conclude our proof. 
\end{proof}

From now on, without loss of generality, we assume that $\fa_1=\O_F$ and $\O_F\subset \fa_i$ for any $i$. Then $F^{\fp,n}\cap\fa_i\subset 1+\fp^n\backslash\{1\}$ if and only if $i\neq 1$. Also we assume that the residue degree of $\fp$ is $1$.
Then by Lemma \ref{gal:av:nonzero:integral}, Lemma \ref{lattice:esti:2} and Lemma \ref{lattice:esti:3}, we can rewrite (\ref{1st:Lvalue}) as
\begin{align}
&\frac{V_{1,\frac{k}{2}}\big(\frac{1}{y}\big)}{\Gamma_{F,\mathbf{k},\mathbf{m}}\big(\frac{k}{2}\big)} \label{1st:Lvalue:main} \\
+&\sum_{(i,\kappa)\neq(1,1)}\sum_{\substack{\a\in (F^{\fp,n}\cap\fa_i^{-1})\cap C_F^0 \\ |N(\a)|>c_FN(\fp)^n }}\frac{a_f(\fa_i\langle\kappa\rangle_n\a)}{N(\fa_i\langle\kappa\rangle_n\a)^{k/2}}\frac{V_{1,\frac{k}{2}}\big(\frac{N(\fa_i\langle\kappa\rangle_n\a)}{y}\big)}{\Gamma_{F,\mathbf{k},\mathbf{m}}\big(\frac{k}{2}\big)}, \label{1st:Lvalue:error}
\end{align}
where (\ref{1st:Lvalue:main}) and (\ref{1st:Lvalue:error}) turn out to be the main term and the error term of (\ref{1st:Lvalue}), respectively.
An estimation of (\ref{1st:Lvalue:error}) can be obtained by following:

\begin{prop}\label{1st:gal:av:Lvalue:error:1} Assume that $\operatorname{f}(F,\fp)=1$. For $\e>0$ and $y>c_F N(\fp)^{n(1+|W|^{-1})}$, we have that
\begin{align*}
(\ref{1st:Lvalue:error}) 
\ll_{\e,\fp} \sum_{\kappa\in W} \big(N(\fp)^n N(\langle\kappa\rangle_n)\big)^{(\theta+\e-\frac{1}{2})}+\frac{y^{\frac{1}{2}+\theta+\e}}{N(\fp)^{n} N(\langle\kappa\rangle_n)}.
\end{align*}
\end{prop}

\begin{proof}
By the bound (\ref{rama:peter:bdd}), (\ref{1st:Lvalue:error}) is bounded by
\begin{align*}
(\ref{1st:Lvalue:error})&\ll_{\e} \sum_{(i,\kappa)\neq(1,1)}\sum_{\substack{\a\in (F^{\fp,n}\cap\fa_i^{-1})\cap C_F^0 \\ |N(\a)|>c_FN(\fp)^n }} \frac{V_{1,\frac{k}{2}}\big(\frac{N(\fa_i\langle\kappa\rangle_n\a)}{y}\big)}{N(\fa_i\langle\kappa\rangle_n\a)^{\frac{1}{2}-\theta-\e}}.
\end{align*}
Let us split the above sum by $\sum_{(i,\kappa)\neq(1,1)}\big((*)+(**)\big)$ where
$$
(*)=\sum_{\substack{\a\in(F^{\fp,n}\cap\fa_i^{-1})\cap C_F^0 \\ c_F N(\fp)^{n}<|N(\a)|\leq y/N(\fa_i\langle\kappa\rangle_n) }},\ (**)=\sum_{\substack{\a\in(F^{\fp,n}\cap\fa_i^{-1})\cap C_F^0 \\ |N(\a)|>y/N(\fa_i\langle\kappa\rangle_n) }}
$$
For $x>0$ and an integral ideal $\fc$ of $F$ which is coprime to $\fp$, define $u_{n,\fc}(x):=\#\{\beta\in(F^{\fp,n}\cap\fc)\cap C_F^0:|N(\beta)|=x\}$. Then $U_{n,\fc}(x)=\sum_{m\leq x} u_{n,\fc}(m)$. By the estimate (\ref{aux.func.esti}) and the Abel summation formula, $(*)$ is bounded by
\begin{align*}
(*)=&\frac{1}{N(\fa_i\langle\kappa\rangle_n)^{\frac{1}{2}-\theta-\e}}\sum_{c_F N(\fp)^{n}<m\leq y/N(\fa_i\langle\kappa\rangle_n)}\frac{u_{1,n}(m)}{m^{\frac{1}{2}-\theta-\e}} \\
\ll&_\e \frac{1}{N(\fa_i\langle\kappa\rangle_n)^{\frac{1}{2}-\theta-\e}}\bigg(\frac{U_{1,n}(N(\fp)^n)}{N(\fp)^{n(\frac{1}{2}-\theta-\e)}}+\frac{U_{1,n}(y/N(\fa_i\langle\kappa\rangle_n))}{(y/N(\fa_i\langle\kappa\rangle_n))^{\frac{1}{2}-\theta-\e}}\\
&+\int_{c_FN(\fp)^n}^{y/N(\fa_i\langle\kappa\rangle_n)}\frac{U_{1,n}(x)}{x^{\frac{3}{2}-\theta-\e}}dx\bigg).
\end{align*}
Set $j\geq 1$. Then similarly, $(**)$ is bounded by
\begin{align*}
(**)=&\frac{y^j}{N(\fa_i\langle\kappa\rangle_n)^{\frac{1}{2}-\theta-\e+j}}\sum_{m>y/N(\fb_n)}\frac{u_{1,n}(m)}{m^{\frac{1}{2}-\theta-\e+j}} \\
\ll&_\e \frac{y^j}{N(\fa_i\langle\kappa\rangle_n)^{\frac{1}{2}-\theta-\e+j}}\bigg(\frac{U_{1,n}(y/N(\fa_i\langle\kappa\rangle_n))}{(y/N(\fa_i\langle\kappa\rangle_n))^{\frac{1}{2}-\theta-\e+j}}+\int_{y/N(\fa_i\langle\kappa\rangle_n)}^{\infty}\frac{U_{1,n}(x)}{x^{\frac{3}{2}-\theta-\e+j}}dx\bigg).
\end{align*}
Note that as $\{\fa_i\}_{i=1}^{h_F}$ is fixed, we can put the terms involved in $N(\fa_i)$ to the implicit constants. 
By the Proposition \ref{lattice:esti:1} and Lemma \ref{lattice:esti:3} on the above equations, then we are done. 
\end{proof}

By using the above propositions, we can give a proof of Theorem \ref{main:thm:1:eq}:

\begin{proof}
Set $y=y_n=N(\fp)^{an}$ for $a>1+|W|^{-1}$, then by the above equation, the main term (\ref{1st:Lvalue:main}), Proposition \ref{1st:gal:av:Lvalue:error:2} and Proposition \ref{1st:gal:av:Lvalue:error:1}, we can prove the theorem.
\end{proof}

\section{Non-vanishing of $L$-values}\label{det}
In this section, we use the algebraicity result of Hida \cite{hida1994critical} and Ramanujan-Petersson bound obtained by Blomer-Brumley \cite{blomer2011ramanujan} and Nakasuji \cite{nakasuji2012generalized} to show the non-vanishing result.

We have the non-vanishing of the special $L$-values as a corollary of Theorem \ref{main:thm:1}:
\begin{cor} Let $\fp$ be a totally split prime ideal of $F$ lying above $p$ and coprime to $h_F\fd_F\fN$. Let $(-1)^{|\Sigma(\R)|(k-2)}C_{F,J,\mathbf{k}}^2=1$. For a newform $f\in S_{(\mathbf{k},\mathbf{m}),J}(\fN,\chi)$, we have 
$$
L\Big(\frac{k}{2},f\otimes\varphi\Big)\neq 0
$$
for almost all Hecke characters $\varphi$ over $F$ of $p$-power orders and $\fp$-power conductors.
\end{cor}

\begin{proof}
To make the error terms of (\ref{main:thm:1:eq})
converges to $0$ when $n$ goes to $\infty$, the numbers $a$ and $\theta$ must satisfy the following inequalities:
\begin{align*}
\theta<\frac{1}{2},\
1+\frac{1}{|W|}<a,\
a\Big(\theta+\frac{1}{2}\Big)<1,\
a\Big(\theta+\frac{1}{2}\Big)>2\theta+\frac{1}{2}
.
\end{align*}
As we have set $\theta\in[0,\frac{1}{2}]$, those are clearly equivalent to
$$
0\leq\theta<\frac{1}{2},\ \frac{2\theta+\frac{1}{2}}{\theta+\frac{1}{2}}<a<\frac{1}{\theta+\frac{1}{2}},\ 1+\frac{1}{|W|}<a<\frac{1}{\theta+\frac{1}{2}} .
$$
Note that $|W|\geq 2$, where the equality holds when $p=3$. So we can find $a$ satisfying the above inequalities if 
\begin{displaymath}
\theta<\min\Big\{\frac{1}{4},\frac{|W|-1}{2(|W|+1)}\Big\}= \left\{ \begin{array}{ll}
1/6 & \text{if } p=3 \\
1/4 & \textrm{Otherwise.}
\end{array} \right.
\end{displaymath}
By Blomer-Brumley \cite[Theorem 1]{blomer2011ramanujan} and Nakasuji \cite[Corollary 1.2]{nakasuji2012generalized}, one has $\theta=7/64<\frac{1}{6}$, hence we can choose such $a$.

Let $\varphi$ be a Hecke character of $p$-power order and conductor $\fp^n$, or a ray class character of $p$-power order and conductor $\fp^n$. Then by Theorem \ref{main:thm:1} and the estimation (\ref{aux.func.esti}), we have
\begin{equation}\label{gal:av:limit}
\lim_{n\rightarrow\infty}L_{\textrm{av}}(f\otimes\varphi)=1
\end{equation}
Hence we have
$$
L_{\textrm{av}}(f\otimes\varphi)=\frac{1}{[K_f(\varphi):K_f]}\sum_{\s\in\text{Gal}(K_f(\varphi)/K_f)} L\Big(\frac{k}{2},f\otimes\varphi^\s\Big)\neq 0
$$
for sufficiently large $n$.
On the other hand, by the algebraicity result of Hida \cite{hida1994critical}, we have
\begin{equation}\label{shimura:rec}
L\Big(\frac{k}{2},f\otimes\varphi\Big)^\s=0 \text{ if and only if } L\Big(\frac{k}{2},f^\s\otimes\varphi^\s\Big)=0
\end{equation} 
for any $\s\in\text{Aut}_\Q(\C)$. Assume that $L\big(\frac{k}{2},f\otimes\varphi\big)=0$ for infinitely many Hecke character $\varphi$ of $p$-power order and $\fp$-power conductor. Then by (\ref{shimura:rec}), we have $L_{\textrm{av}}(f\otimes\varphi)=0$ for infinitely many such $\varphi$. This contradicts to the equation (\ref{gal:av:limit}).
\end{proof}








\section*{Acknowledgements}
\thispagestyle{empty}

This research was supported by the Basic Science Research Program through the
National Research Foundation of Korea(NRF) funded by the Ministry of Education(NRF-
2017R1A2B4012408). The authors would like to thank Ashay Burungale, Keunyoung Jeong and Chan-Ho Kim for helpful discussions and comments.

\end{document}